\documentclass[preprint,1p,10pt]{elsarticle}

\journal{arXiv}

\usepackage{color,soul}
\biboptions{sort&compress,comma,round}
\usepackage{stmaryrd}
\usepackage{graphicx}
\usepackage{amssymb}
\usepackage{bm}
\usepackage[FIGTOPCAP]{subfigure}
\usepackage{amsmath}
\usepackage{float}
\DeclareSymbolFont{largesymbolsA}{U}{txexa}{m}{n}
\DeclareMathSymbol{\bigtimes}{\mathop}{largesymbolsA}{16}

\usepackage{color}
\usepackage[table]{xcolor}
\usepackage{setspace}
\usepackage{graphics}
\usepackage{amssymb, latexsym, amsmath, epsfig}
\usepackage{graphicx}
 \usepackage{color}
 \usepackage{epstopdf}
 \usepackage{comment}
 \usepackage{soul}
 \usepackage[normalem]{ulem}
  \usepackage{float}
\usepackage{amssymb,amsmath}
\usepackage{amsfonts,amstext,amsthm,amssymb,comment}
\usepackage{epsfig,graphics,color}

\graphicspath{{figures/}}

\newtheorem{theorem}{Theorem}

\numberwithin{equation}{section}

\newtheorem{remark}{Remark}

\newcommand{\bfs}[1]{{\boldsymbol #1}}

\newcommand{\Question}[1]{\marginpar{}}
\renewcommand{\Question}[1]{%
            \marginpar{\flushleft\scriptsize\bfseries\upshape#1}}

\begin{document}

\begin{frontmatter}


\title{Higher-order generalized-$\alpha$ methods for hyperbolic problems}

%
\author[ad1]{Pouria Behnoudfar\corref{corr}}
\ead{pouria.behnoudfar@postgrad.curtin.edu.au}
\cortext[corr]{Corresponding author}
\author[ad1]{Quanling Deng}
\ead{Quanling.Deng@curtin.edu.au}

\author[ad1,ad2]{Victor M. Calo}
\ead{Victor.Calo@curtin.edu.au}
%

%
%
%
%
%
\address[ad1]{Applied Geology, School of Earth and Planetary Sciences, Curtin University, Kent Street, Bentley, Perth, WA 6102, Australia}
\address[ad2]{Mineral Resources, Commonwealth Scientific and Industrial Research Organisation (CSIRO), Kensington, Perth, WA 6152, Australia}
%

\begin{abstract}
The generalized-$\alpha$ time-marching method provides second-order accuracy in time and controls the numerical dissipation in the high-frequency region of the discrete spectrum. This method includes a wide range of time integrators. We increase the order of accuracy of the method while keeping the unconditional stability and the user-control on the high-frequency numerical dissipation. The dissipation is controlled by a single parameter as in the original method. Our high-order schemes require simple modifications of the available implementations of the generalized-$\alpha$ method.
\end{abstract}

\begin{keyword}
generalized-$\alpha$ method \sep high-order time integration \sep spectrum analysis \sep hyperbolic equation \sep dissipation control \sep stability analysis

\end{keyword}

\end{frontmatter}

\vspace{0.5cm}

\section{Introduction}
Chung and Hulbert in \cite{chung1993time} introduced the generalized-$\alpha$ method for solving hyperbolic equations arising in structural dynamics. The method has second-order accuracy in time, unconditional stability, and user-control on the high-frequency numerical dissipation. Consequently, the method has been widely used for various applications.\\
The Newmark-$\beta$ method introduced in \cite{newmark1959method} is second-order accurate in time. However, the numerical dissipation cannot be controlled and the numerical solution is too dissipative in the low-frequency region. The generalized-$\alpha$ method improves on the $\phi$ method of Wilson~\cite{wilson1968computer}, the $\phi_1$ method of Hoff and Pahl \cite{hoff1988development}, and the $\rho$ method of Bazzi and Anderheggen \cite{bazzi1982rho}. These methods are second-order accurate and attain high-frequency dissipation but still have some low-frequency damping.\\ 

The  generalized-$\alpha$ method generalizes the well-known HHT-$\alpha$ method of Hilber, Hughes, Taylor \cite{hilber1977improved} and the WBZ-$\alpha$ method of Wood, Bossak, and Zienkiewicz \cite{wood1980alpha}. That is, setting the parameters in generalized-$\alpha$ method to particular values, the method reduces to either the HHT-$\alpha$ or WBZ-$\alpha$ methods. The generalized-$\alpha$ method produces an algorithm which provides an optimal combination of high-frequency and low-frequency dissipation in the sense that for a given value of high-frequency dissipation, the algorithm minimizes the low-frequency dissipation; see \cite{chung1993time}. \\
To the authors' best knowledge, all these methods including the generalized-$\alpha$ method are limited to second order accuracy in time while the high-order Lax-Wendroff, Runge-Kutta, Adams-Moulton, and backward differentiation schemes (see \cite{butcher2016numerical}) lack the explicit control over the numerical dissipation of the high frequencies. Thus, we propose a $k$-step form of the generalized-$\alpha$ method that delivers $2k$ accuracy in time for problems of second derivatives in time. The main idea of our generalization is to add higher-order terms as the residuals obtained by solving auxiliary systems as well as adopting higher-order terms in Taylor expansions used in the generalized-$\alpha$ method. More precisely, to gain $2k$ order of accuracy, we build an algorithm that consists of $3n$ equations. For each set of three equations, we solver a system for a variable and update the other two explicitly. We then study the spectral properties of the resulting amplification matrix to determine the unconditional stability region. The resulting system has control over the high-frequency numerical dissipation as well as the same unconditional stability region as the second-order scheme. The revisions to an implemented generalize-$\alpha$ code are simple and the method remains highly efficient. 
The rest of this paper is organized as follows. Section 2 describes the problem under consideration. Section 3 presents the main idea of the fourth-order generalized-$\alpha$ method. We prove the fourth-order accuracy in time and the unconditional stability region. Section 4 discusses our method for $2k$-order accuracy by introducing the unconditional stability region as well as the numerical dissipation control parameters.  
Concluding remarks are given in Section 5.

\section{Problem Statement} \label{sec:ho}
We consider the second-order ordinary differential equation (ODE)
\begin{equation} \label{eq:ode}
\begin{aligned}
\ddot u + \lambda u& = 0,\\
u(0)&=u_0,\\
\dot{u}(0)&=v_0,
\end{aligned}
\end{equation}
where $u_0$ is the initial solution and $v_0$ is the first derivative of solution at the initial state. We choose the time marching intervals between $0 = t_0 < t_1 < \cdots < t_N = T$ where $T$ is the final time and define the time step-size as $\tau_n = t_{n+1} - t_n$. We also denote the approximation of $U(t_n), \dot U(t_n), \ddot U(t_n)$, by $U_n, V_n, A_n$ respectively. The generalized-$\alpha$ method for solving \eqref{eq:ode} at time-step $n$ is given by: 
\begin{subequations} \label{eq:ga}
\begin{align}
\label{eq:ga1}
A_{n+\alpha_m}&=-\lambda U_{n+\alpha_f}, \\
\label{eq:ga2}
U_{n+1}  &= U_n + \tau v_n + \frac{\tau^2}{2} A_n + \tau^2 \beta_1 \llbracket A_n \rrbracket, \\
\label{eq:ga3}
V_{n+1} & = V_n + \tau A_n + \tau \gamma_1 \llbracket A_n \rrbracket, 
\end{align}
\end{subequations}
with the initial solution $U_0 = u_0$, the initial velocity $V_0 = v_0$, the initial acceleration $A_0=-\lambda U_0$, and
\begin{equation} \label{eq:param}
\begin{aligned}
w_{n+\alpha_g} & = w_n + \alpha_g \llbracket w_n \rrbracket,\qquad \llbracket w_n \rrbracket & = w_{n+1} - w_n, \quad w = U, A, \quad g=m, f.\\
\end{aligned}
\end{equation}
When $\gamma_1 = \frac{1}{2} + \alpha_m - \alpha_f$ and  $\beta_1 = \frac{1}{4}( 1+\alpha_m - \alpha_f)^2$, this reduces to the generalized-$\alpha$ method to solve the hyperbolic problems; see \cite{chung1993time}. In order to control the numerical dissipation, the following parameter definitions are used
\begin{equation}
	\alpha_f=\frac{1}{1+\rho_\infty}, \qquad \alpha_m=\frac{2-\rho_\infty}{1+\rho_\infty},
\end{equation}
where, $\rho_\infty\in[0,1]$ is the user-defined control parameter. 
\section{Fourth-order generalized-$\alpha$ method} \label{sec:ho}
The equations \eqref{eq:ga2} and \eqref{eq:ga3} imply a sub-step time-marching for the generalized-$\alpha$ method and affect the order of accuracy. To obtain these two equations we use a Taylor expansion, and as a result they limit the accuracy of the method as they imply truncation error of $\mathcal{O}(\tau^3)$. 
To overcome this limitation, we derive these representations by applying a higher-order accurate Taylor expansion; see also the discussion in \cite{deng2019high}. For this purpose, let $\mathcal{L}^a(s)$ denotes the $a$-th order derivative of the function $s$ in time. Thus, for instance, to derive a fourth-order generalized-$\alpha$ method, we propose a method based on solving
\begin{equation} \label{eq:4a}
\begin{aligned}
A_{n}^{\alpha_1}&=-\lambda U_{n+1}, \\
\mathcal{L}^{3}(A_n^{\alpha_2})&=-\lambda \mathcal{L}^{1}(A_n^{\alpha_f}), \\
\end{aligned}
\end{equation}
with updating conditions
\begin{equation} \label{eq:4aup}
\begin{aligned}
U_{n+1}  &= U_n + \tau V_n + \frac{\tau^2}{2} A_n + \frac{\tau^3}{6} \mathcal{L}^{1}(A_n)+ \frac{\tau^4}{24}\mathcal{L}^{2}(A_n)+ \frac{\tau^5}{120} \mathcal{L}^{3}(A_n)+\beta_1 \tau^2 P_{n}, \\
V_{n+1}  &= V_n + \tau A_n + \frac{\tau^2}{2} \mathcal{L}^{1}(A_n) + \frac{\tau^3}{6} \mathcal{L}^{2}(A_n)+ \frac{\tau^4}{24} \mathcal{L}^{3}(A_n)+\gamma_1 \tau P_{n}, \\[0.2cm]
\mathcal{L}^{1}(A_{n+1})&= \mathcal{L}^{1}(A_{n}) + \tau \mathcal{L}^{2}(A_n) + \frac{\tau^2}{2}\mathcal{L}^{3}(A_n) + \tau^2 \beta_2  \llbracket \bm{\mathcal{L}^{3}(A_{n})} \rrbracket, \\
\mathcal{L}^{2}(A_{n+1})&= \mathcal{L}^{2}(A_n) + \tau \mathcal{L}^{3}(A_n) + \tau \gamma_2 \llbracket \bm{\mathcal{L}^{3}(A_{n})} \rrbracket,\\
\end{aligned}
\end{equation}
where
\begin{equation} 
\begin{aligned}
P_{n}&=A_{n+1}-A_n-\tau \mathcal{L}^{1}(A_n)-\frac{\tau^2}{2} \mathcal{L}^{2}(A_n)-\frac{\tau^3}{6} \mathcal{L}^{3}(A_n),\\
A_{n}^{\alpha_1}&=A_n+\tau \mathcal{L}^{1}(A_{n})+\frac{\tau^2}{2} \mathcal{L}^{2}(A_{n})+\frac{\tau^3}{6} \mathcal{L}^{3}(A_{n})+\alpha_1 P_{n},\\[0.2cm]
\mathcal{L}^{3}(A_n^{\alpha_2})&=\mathcal{L}^{3}(A_n)+\alpha_2\llbracket \bm{\mathcal{L}^{3}(A_{n})} \rrbracket,\\
\mathcal{L}^{1}(A_n^{\alpha_f})&=\mathcal{L}^{1}(A_n)+\alpha_f \llbracket \bm{\mathcal{L}^{1}(A_{n})} \rrbracket.\\
\end{aligned}
\end{equation}
Assuming sufficient smoothness of the solution on the time interval under analysis and by taking three derivatives from the first equation of \eqref{eq:4a} with respect to time, we readily obtain $\mathcal{L}^{3}(A_n^{\alpha_2})=-\lambda \mathcal{L}^{1}(A_n^{\alpha_f})$.
 The initial data are also obtained by using the given information on $U_0$ and $V_0$ as
 \begin{equation}
 \begin{aligned}
  A_0&=-\lambda U_0,\qquad \mathcal{L}^{1}(A_0)&=-\lambda V_0, \\
  \mathcal{L}^{2}(A_0)&=\lambda^2 U_0, \qquad \mathcal{L}^{3}(A_0)&=\lambda^2 V_0.
\end{aligned}
 \end{equation}
 
 \subsection{Order of accuracy in time}
 Herein, we determine the conditions on the parameters $\gamma_1$ and $\gamma_2$ to guarantee the fourth-order accuracy of the scheme in the form of \eqref{eq:4a}. We have the following result.
 
 \begin{theorem} \label{thm:3o}
Assuming that the solution is sufficiently smooth with respect to time, the method in \eqref{eq:4a} is fourth-order accurate in time given
 	\begin{equation} \label{eq:3ov1}
 	\gamma_1=\alpha_1-\frac{1}{2}, \qquad \qquad 	\gamma_2=\frac{1}{2}-\alpha_{f}+\alpha_2.
 	\end{equation}
 \end{theorem}
 \begin{proof}
Substituting \eqref{eq:4aup} into \eqref{eq:4a}, we obtain a system of equations for each time step as
 \begin{equation}
 A\bold{ U_{n+1}}=B \bold{U_n},
 \end{equation}
 where 
\begin{equation}
\begin{aligned}A&=
\begin{bmatrix}
1 & 0 & -\beta_1&0&0&0 \\
0 & 1 & -\gamma_1&0&0&0 \\
\tau^2 \lambda & 0 & \alpha_1&0&0&0 \\
0 & 0 & 0&1&0&-\beta_2 \\
0 & 0 & 0&0&1&-\gamma_2 \\
0 & 0 & 0&\tau^2\alpha_f \lambda&0&\alpha_2 \\
\end{bmatrix},
\\
B&=
\begin{bmatrix}
1 & 1 & \frac{1}{2}-\beta_1& \frac{1}{6}-\beta_1& \frac{1}{24}-\frac{\beta_1}{2}& \frac{1}{120}-\frac{\beta_1}{6} \\
0 & 1 & 1-\gamma_1& \frac{1}{2}-\beta_1& \frac{1}{6}-\frac{\beta_1}{2}& \frac{1}{24}-\frac{\beta_1}{6} \\
0  & 0 & \alpha_1-1& \alpha_2-1&\frac{1}{2}( \alpha_2-1)&\frac{1}{6}(\alpha_2-1) \\
0 & 0 & 0&1&1&\frac{1}{2}-\beta_2 \\
0 & 0 & 0&0&1&1-\gamma_2 \\
0 & 0 & 0&-\tau^2(1-\alpha_{f}) \lambda&0&\alpha_2-1 \\
\end{bmatrix},\\
\bold{U_{n}}&=
\begin{bmatrix}
U_{n}  \\
\tau V_{n}  \\
\tau^2 A_{n}\\
\tau^3\mathcal{L}^{1}(A_{n})\\
\tau^4 \mathcal{L}^{2}(A_{n})\\
\tau^5 \mathcal{L}^{3}(A_{n})\\
\end{bmatrix}.
\end{aligned}
\end{equation}
Thus, the amplification matrix $G$ is
\begin{equation} \label{eq:ampm}
G=A^{-1}B.
\end{equation}
This matrix-matrix multiplication results in upper-block diagonal matrix. The high-order unknowns $\mathcal{L}^{1}(A_{n}),\mathcal{L}^{2}(A_{n}),\mathcal{L}^{3}(A_{n})$, associated with the lower block on the diagonal, are required to be second-order accurate. The upper block on the diagonal also leads to second-order accuracy. Then, we obtain the fourth-order accuracy by adding the high-order terms through the upper off-diagonal block to our solution. Thus, we analyze each block to introduce corresponding parameters. For any arbitrary amplification matrix, we can state the following
\begin{equation} \label{eq:a40}
G_0 \mathcal{L}^{1}(A_{n+1}) - G_1 \mathcal{L}^{1}(A_n) + G_2 \mathcal{L}^{1}(A_{n-1}) - G_3 \mathcal{L}^{1}(A_{n-2}) = 0,
\end{equation}
where the coefficients are invariants of the amplification matrix as $G_0 = 1$, $G_1 $ is the trace of $G$, $G_2$ is the sum of principal minors of $G$, and $G_3$ is the determinant of $G$. By using a Taylor series expansion, we obtain
\begin{equation} \label{eq:te3}
\begin{aligned}
\mathcal{L}^{1}(A_{n+1}) & = \mathcal{L}^{1}(A_{n})+\tau \mathcal{L}^{2}(A_{n}) + \frac{\tau^2}{2} \mathcal{L}^{3}(A_{n})+ \mathcal{O}(\tau^3), \\
\mathcal{L}^{1}(A_{n-1}) & = \mathcal{L}^{1}(A_{n})-\tau \mathcal{L}^{2}(A_{n}) + \frac{\tau^2}{2} \mathcal{L}^{3}(A_{n})+ \mathcal{O}(\tau^3), \\
\mathcal{L}^{1}(A_{n-2}) & = \mathcal{L}^{1}(A_{n}) -2\tau \mathcal{L}^{2}(A_{n}) + {2\tau^2} \mathcal{L}^{3}(A_{n})+ \mathcal{O}(\tau^3).
\end{aligned}
\end{equation}
Setting $\gamma_2=\frac{1}{2}-\alpha_{f}+\alpha_2$, \eqref{eq:te3} is second-order accurate in time. Then, we consider the other three equations in \eqref{eq:4a} and \eqref{eq:4aup}. Then, we have
\begin{equation} \label{eq:te6}
\begin{aligned}
U_{n+1} & = U_n + \tau V_n + \frac{\tau^2}{2} A_n+\mathcal{R} , \\
U_{n-1} & = U_n - \tau V_n + \frac{\tau^2}{2} A_n+\mathcal{R} , \\
U_{n-2} & = U_n - 2\tau V_n + 4\frac{\tau^2}{2} A_n +\mathcal{R}, \\
\end{aligned}
\end{equation}
where the term $\mathcal{R}$ defined as a function of $\mathcal{L}^{1}(A_{n}), \mathcal{L}^{2}(A_{n})$ and  $\mathcal{L}^{3}(A_{n})$. This term is a  residual and we neglected it in the analysis. Thus, by following a similar approach, we can verify that the remaining terms are of the second-order accuracy in time. Then, we add the residuals to the second-order accurate solution, in order to have the truncation error of $\mathcal{O}(\tau^5)$ and consequently, a fourth-order accurate scheme in time. This completes the proof.

\end{proof}
\subsection{Stability analysis and eigenvalue control}\label{sec:anal}
In order to have an unconditionally stable method, we bound the absolute values of the eigenvalues of amplification by one. For this purpose,by considering $\lambda\tau^2$ as one term, we first calculate the eigenvalues of \eqref{eq:ampm} for the case $\sigma=\lambda\tau^2 \to 0$ as
\begin{equation}\label{eq:t0}
	\lambda_1=\lambda_2=\lambda_3=\lambda_4=1,\qquad \lambda_5=\frac{\alpha_1-1}{\alpha_1}, \qquad \lambda_6=\frac{\alpha_2-1}{\alpha_2}.
\end{equation}
The boundedness of $\lambda_5$ and $\lambda_6$ in \eqref{eq:t0} implies the conditions $ {\alpha_1}\geq \frac{1}{2} $ and $ {\alpha_1} \geq \frac{1}{2} $.
For the case of $\sigma \to \infty$, we show in Figure \ref{fig:eig} the stability region for two cases in which $\alpha_1$ is constant as well as when it is equal to $\alpha_1$. Hence, to achieve unconditional stability, the corresponding parameters can be set to
\begin{equation}
1 \leq {\alpha_1}, \qquad \frac{1}{2} \leq {\alpha_f} \leq {\alpha_2}.
\end{equation}
\begin{figure}[!ht]	
	\centering\includegraphics[width=6.5cm]{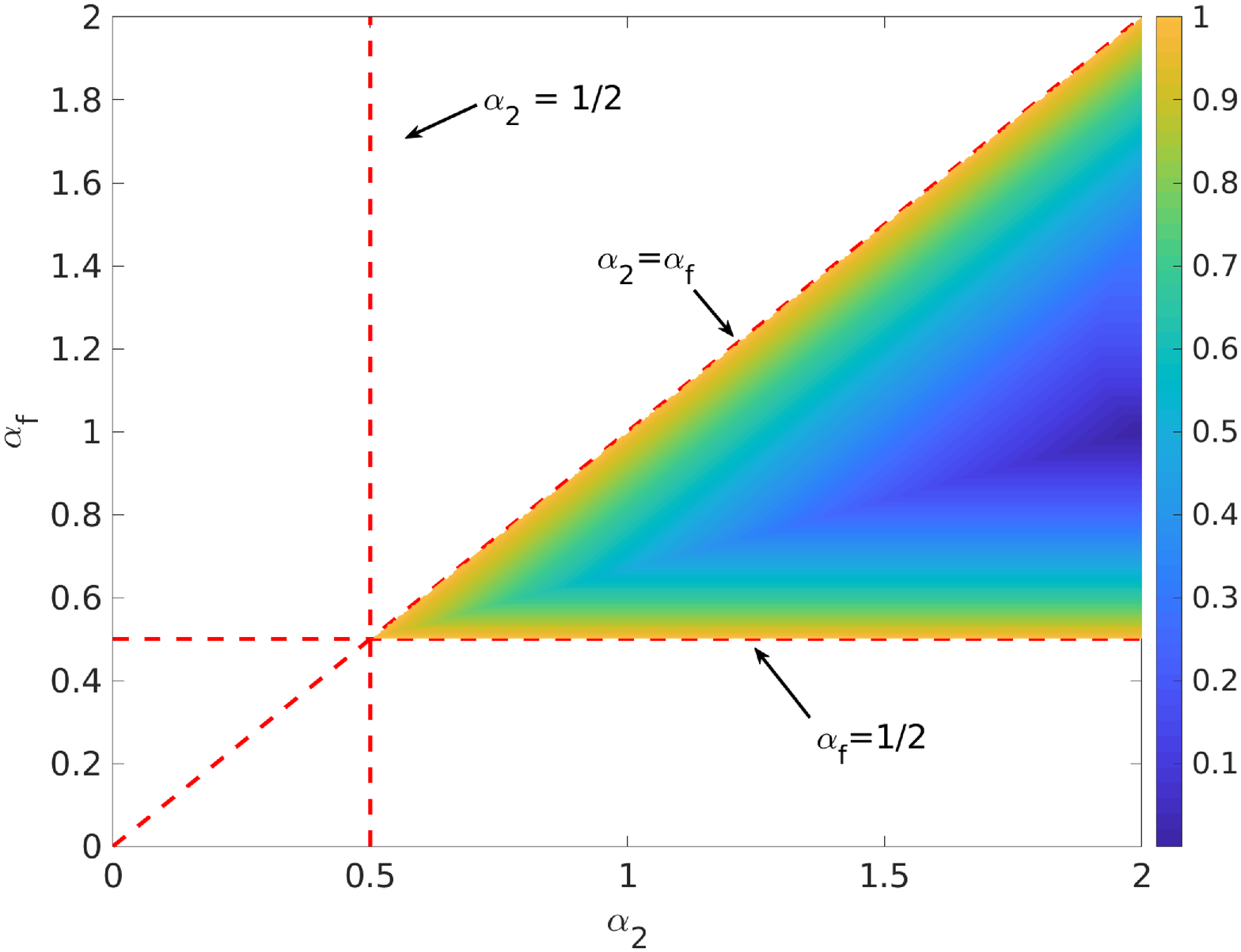} 
	\hspace{0.15 cm}
	\centering\includegraphics[width=6.5cm]{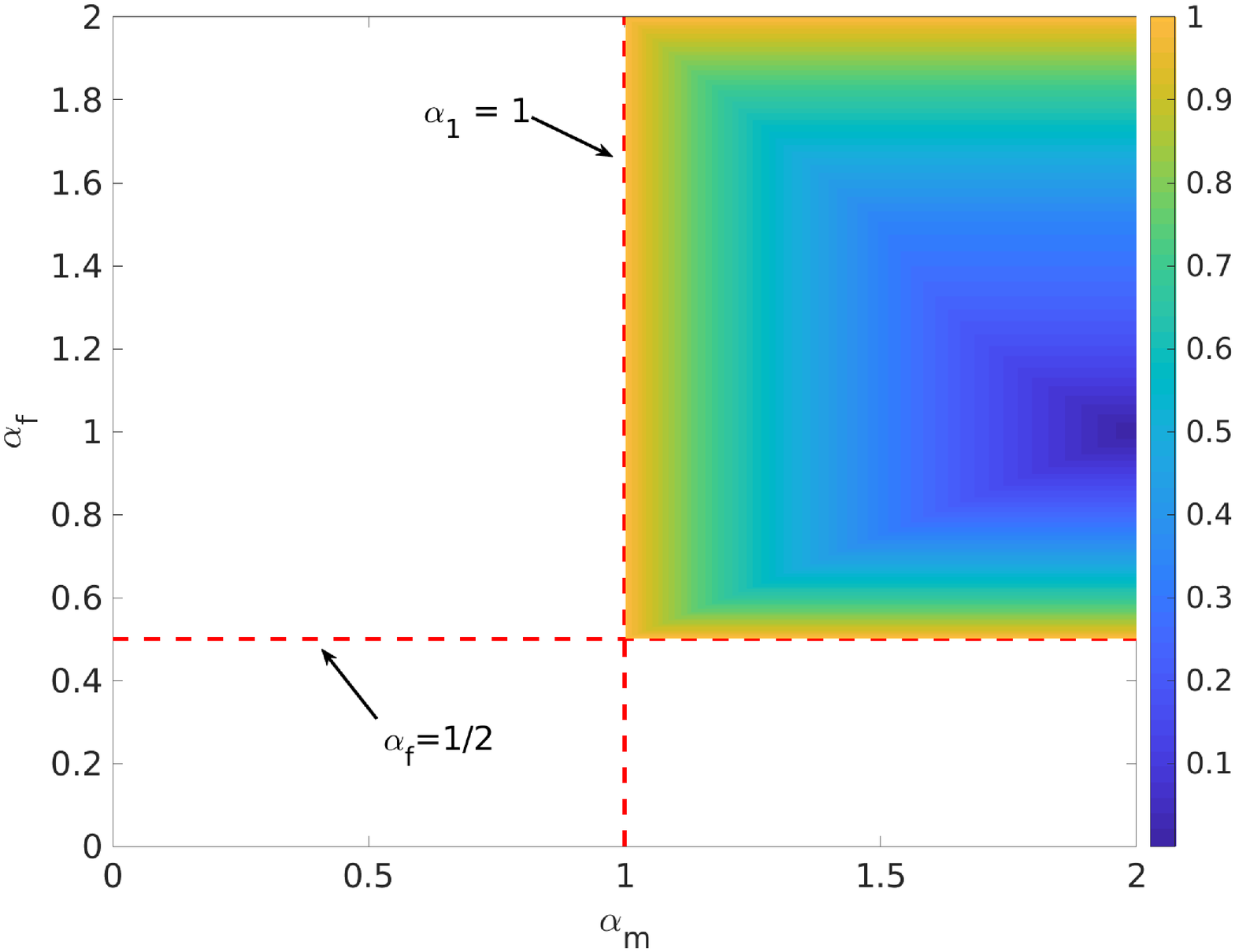} 	
	\caption{The stability region for the cases $\alpha_1=2$ on the left, and $\alpha_2=2$ on the right.}
	\label{fig:eig}
\end{figure}
To provide the control on the numerical dissipation, following closely the analysis on the second-order generalized-$\alpha$ method in \cite{chung1993time,jansen2000generalized,behnoudfar2018variationally}, we use two other parameters $\beta_1$ and $\beta_2$ to set the complex part of the eigenvalues equal to zero when $\sigma \to \infty$. Thus, we have
 \begin{equation}
 	\beta_1=\frac{1}{16}\big(1+4\gamma_1+4\gamma_1^2\big), \qquad \beta_2=\frac{1}{16}\big(1+4\gamma_2+4\gamma_2^2\big).
 \end{equation}
Then, we set user-controlled parameters $\rho^\infty_1$, $\rho^\infty_2$ and use the following definitions to propose the method: 
\begin{equation}
\begin{aligned}
\alpha_1=\frac{2}{1+\rho^\infty_1}, \qquad \alpha_{f}=\frac{1}{1+\rho^\infty_2}, \qquad \alpha_2=\frac{2-\rho^\infty_2}{1+\rho^\infty_2}.
\end{aligned}
\end{equation}
\begin{remark}
	The amplification matrix when $\sigma \to \infty$ is a block matrix and it has two sets of eigenvalues. Two of the eigenvalues equal to $\rho^\infty_1$, one is zero and the other three equal to $\rho^\infty_2$. We omit the details for brevity.
\end{remark}
Therefore, by choosing $0\leq \rho^\infty_1,\rho^\infty_2\leq1 $, one controls the eigenvalues of the amplification matrix and the high-frequency damping. We show this in Figure \ref{fig:dis} where $\sigma=-\lambda \tau^2$. For large $\sigma$, the eigenvalues $\lambda_{1,2,3}$ corresponding to the first block of the amplification matrix approach $0$ and $\rho^\infty_1$ and eigenvalues of the second  block $\lambda_{4,5,6}$ reach to $\rho^\infty_2$.
\begin{figure}[!ht]	
	\centering\includegraphics[width=6.3cm]{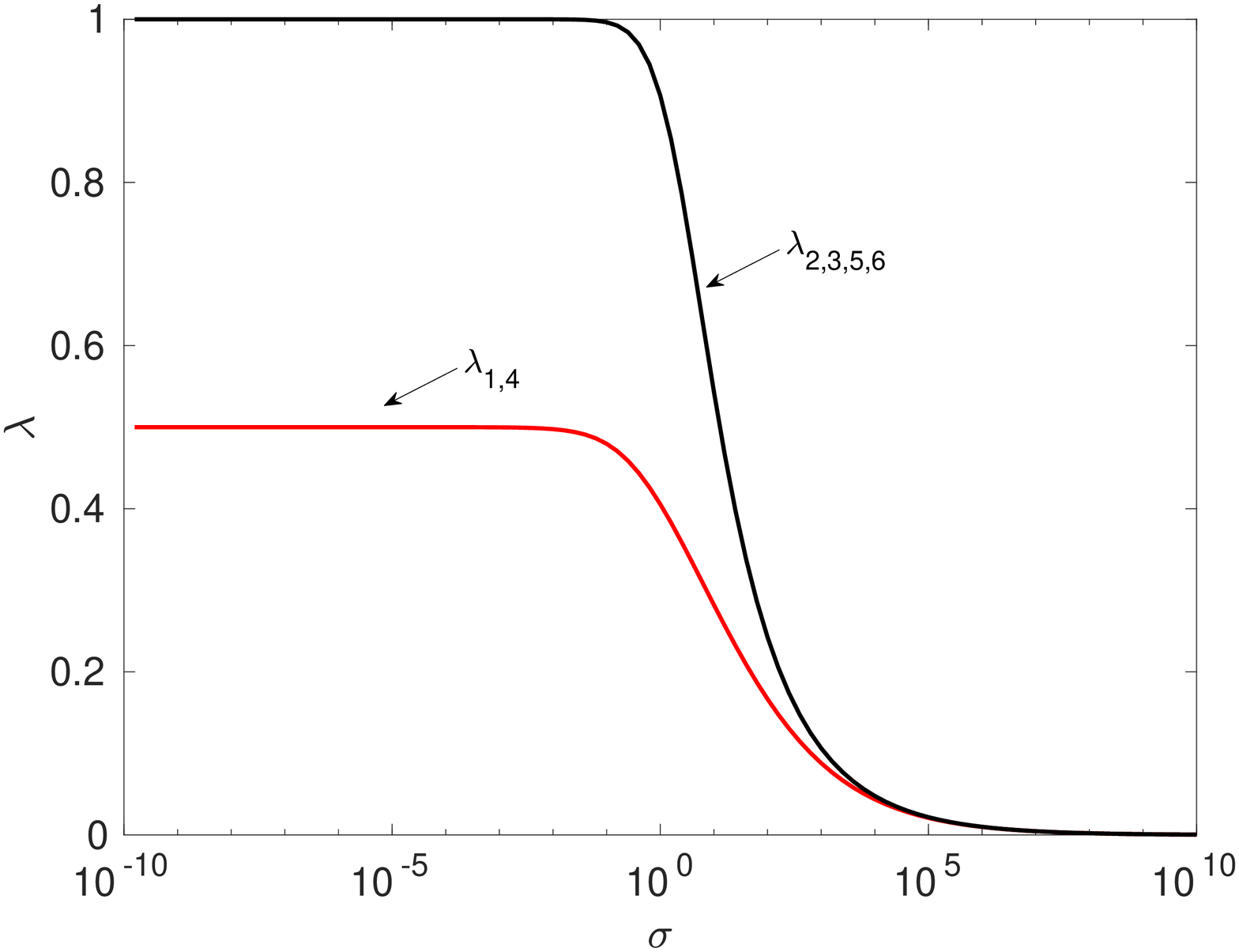} 
	\centering\includegraphics[width=6.3cm]{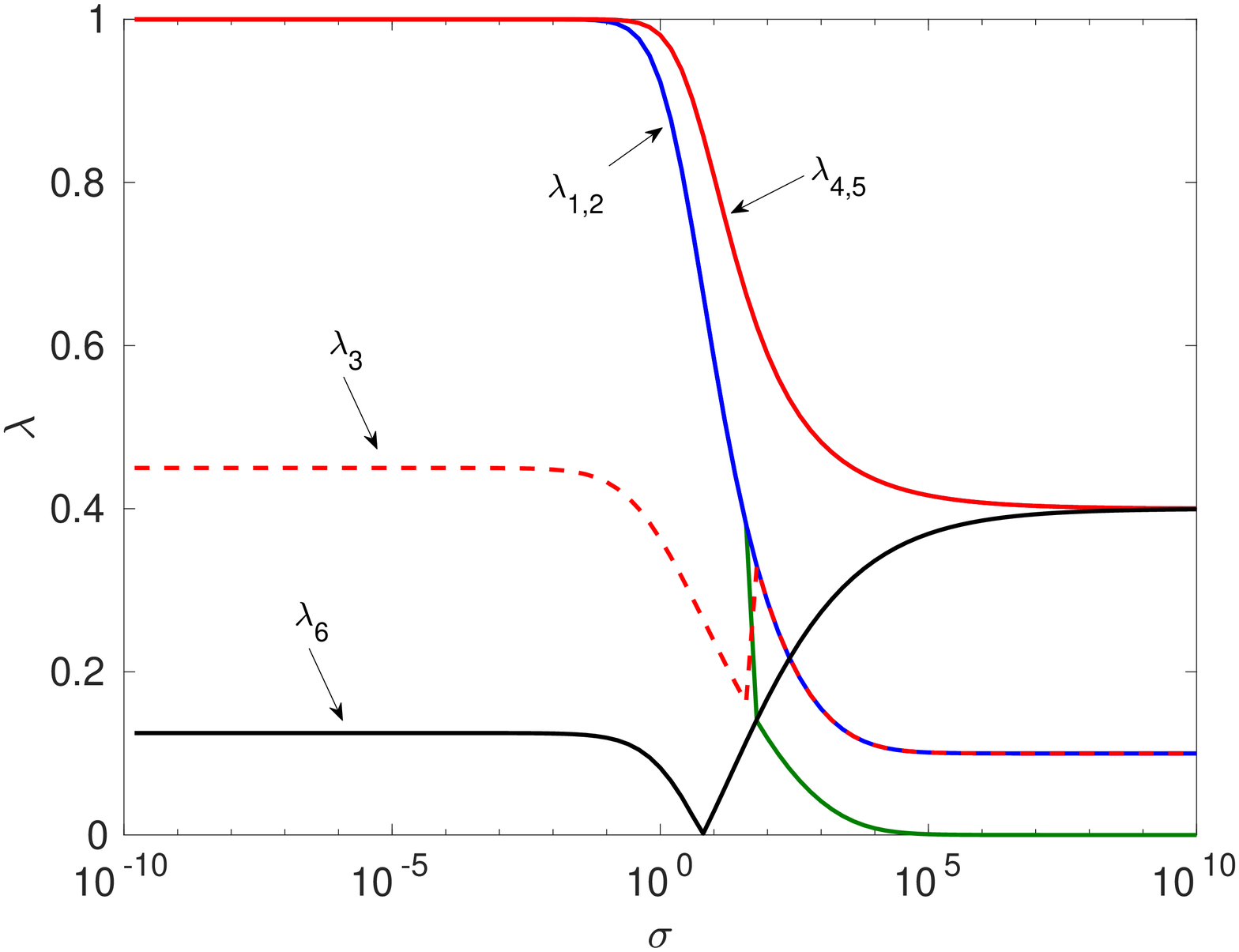} 	
	\caption{The eigenvalues of the amplification matrix belong to the fourth-order accuracy time-marching scheme when $ \rho^\infty_1= \rho^\infty_2=0$ on the left, and $\rho^\infty_1=0.1$ and $ \rho^\infty_2=0.4$ on the right.}
	\label{fig:dis}
\end{figure}

\section{Higher-order accuracy in time}

In general, for $k\ge 2$, we obtain the $2k$-th order generalized-$\alpha$ method by solving
\begin{equation} \label{eq:ho1}
\begin{aligned}
A_{n}^{\alpha_1}&=-\lambda U_{n+1}, \\
\mathcal{L}^{3j-3}(A_n^{\alpha_j})&=-\lambda \mathcal{L}^{3j-5}(A_{n+1}),\qquad j=2, \cdots, k-1, \\
\mathcal{L}^{3k-3}(A_n^{\alpha_k})&=-\lambda \mathcal{L}^{3k-5}(A_{n}^{\alpha_{f}}), \\
\end{aligned}
\end{equation}
and updating the system using the following
\begin{equation} \label{eq:ho2}
\begin{aligned}
U_{n+1} & = U_n + \tau v_n + \frac{\tau^2}{2} A_n + \frac{\tau^3}{6} \mathcal{L}^{1}(A_n)+ \cdots+\frac{\tau^{3k-3}}{(3k-3) !} \mathcal{L}^{3k-3}(A_n)+\beta_1 \tau^2 P_{n,1}, \\
V_{n+1}  &= V_n + \tau A_n + \frac{\tau^2}{2}  \mathcal{L}^{1}(A_n) + \cdots+\frac{\tau^{3k-4}}{(3k-4) !} \mathcal{L}^{3k-3}(A_n)+ \tau \gamma_1 P_{n,1}, \\
\mathcal{L}^{3j-5}(A_{n+1})&= \mathcal{L}^{3j-5}(A_n) + \tau \mathcal{L}^{3j-4}(A_n) + \frac{\tau^2}{2}\mathcal{L}^{3j-3}(A_n) + \tau^2 \beta_j P_{n,j}, \\
\mathcal{L}^{3j-4}(A_{n+1})&= \mathcal{L}^{3j-4}(A_{n}) + \tau \mathcal{L}^{3j-3}(A_{n}) + \tau \gamma_j P_{n,j},\qquad j=2, \cdots, k-1, \\[0.2cm] 
\mathcal{L}^{3k-5}(A_{n+1})&= \mathcal{L}^{3k-5}(A_n) + \tau {L}^{3k-4}(A_n) + \frac{\tau^2}{2}\mathcal{L}^{3k-3}(A_n) + \tau^2 \beta_k  \llbracket \mathcal{L}^{3k-3}(A_{n} \rrbracket, \\
\mathcal{L}^{3k-4}(A_{n+1}) &=\mathcal{L}^{3k-4}(A_{n}) + \tau \mathcal{L}^{3k-3}(A_{n}) + \tau \gamma_k \llbracket \mathcal{L}^{3k-3}(A_{n} \rrbracket,\\
\end{aligned}
\end{equation}
where we have
\begin{equation} \label{eq:ho3}
\begin{aligned}
P_{n,1}&=A_{n+1}-A_n-\tau\mathcal{L}^{1}(A_{n})-\cdots-\frac{\tau^{3k-5}}{(3k-5) !} \mathcal{L}^{3k-3}(A_{n}),\\
A_{n}^{\alpha_1}&=A_n+\tau \mathcal{L}^{1}(A_{n})+\cdots+\frac{\tau^{3k-3}}{(3k-3) !} \mathcal{L}^{3k-3}(A_{n})+\alpha_1 P_{n,1},\\
P_{n,j}&=\mathcal{L}^{3j-3}(A_{n+1})-\mathcal{L}^{3j-3}(A_n)-\cdots-\frac{\tau^3}{6} \mathcal{L}^{3j}(A_n),\\
A_{n}^{\alpha_j}&=\mathcal{L}^{3j-3}(A_n)+\cdots+\frac{\tau^3}{6} \mathcal{L}^{3j}(A_n)+\alpha_jP_{n,j}, \qquad j=2, \cdots, k-1,\\
\mathcal{L}^{3k-3}(A_n^{\alpha_k})&=\mathcal{L}^{3k-3}(A_n)+\alpha_k \llbracket \mathcal{L}^{3k-3}(A_{n} \rrbracket,\\
\mathcal{L}^{3k-5}(A_n^{\alpha_f})&=\mathcal{L}^{k3-5}(A_n)+\alpha_f \llbracket \mathcal{L}^{3k-5}(A_{n} \rrbracket. \\
\end{aligned}
\end{equation}
For $k=1, 2$, this reduces to the second- and fourth-order generalized-$\alpha$ methods, respectively.
Using a similar argument to the one we described in the proof of Theorem \ref{thm:3o}, we can establish higher-order schemes in the form of \eqref{eq:ho1} and \eqref{eq:ho2}. 
To seek $p$-th order  ($p\ge 4$) scheme,  we substitute \eqref{eq:ho2} into \eqref{eq:ho1} and obtain a system written in a matrix form
\begin{equation}\label{eq:eig2}
L \bfs{U}_{n+1} = R \bfs{U}_n.
\end{equation}
Therefore, the amplification matrix is $L^{-1}R$. 
\begin{theorem}
For $u$ that is sufficiently smooth in time, the scheme defined by \eqref{eq:ho1}-\eqref{eq:ho3} for the ODE \eqref{eq:ode} is of $2k$-order accurate in time provided
	\begin{equation} \label{eq:3ov1}
	\begin{aligned}
\gamma_i&=\alpha_i-\frac{1}{2}, \qquad \text{for } i=1,\cdots,k-1,\\
\gamma_i&=\frac{1}{2}-\alpha_{f}+\alpha_k.
	\end{aligned}
\end{equation}
\begin{proof}
	The amplification matrix corresponding to the scheme is a block matrix shown in Figure \ref{fig:shem} corresponding $k=5$. Each block is a $3 \times 3$ matrix where the blue blocks are zero. The green blocks also have similar entries to the entries of the amplification used in the second-order generalized-$\alpha$ method. Hence, the terms used in \eqref{eq:a40} can be calculated separately for each block and considering the higher-order terms to have a second-order accuracy. Consequently, after solving the whole system, we have a truncation error of $\mathcal{O}(\tau^{2k+1})$.
\begin{figure}[!ht]
	\centering\includegraphics[width=6.5cm]{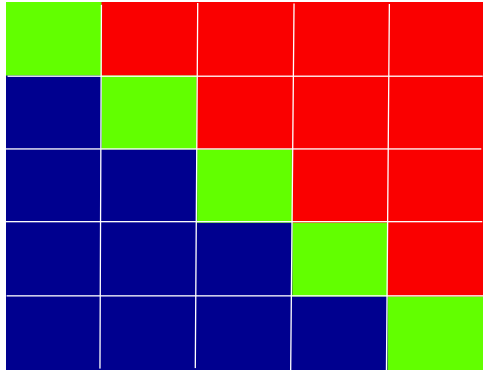} 	
	\caption{The amplification corresponding to the method with $2k$-order accurate in time. The blue blocks are zeros and green ones have the structure similar to the second-order generalized-$\alpha$ method.}
	\label{fig:shem}
\end{figure}
\end{proof}
\end{theorem}

\subsection{Stability analysis and control on dissipation}
To study the eigenvalues of the amplification matrix \eqref{eq:eig2}, we refer to the sketch that Figure \ref{fig:shem} presents, where we can obtain all the eigenvalues by calculating the eigenvalues of each green block. Hence, we propose the following for the parameters (follow a similar logic to the one we describe in Section \ref{sec:anal})
\begin{equation}
\beta_i=\frac{1}{16}\big(1+4\gamma_i+4\gamma_i^2\big), \qquad i=1,\cdots,k.
\end{equation}
Similarly, we define user-controlled parameters $\rho_\infty^i$ and set
\begin{equation}
\begin{aligned}
\alpha_i&=\frac{2}{1+\rho^\infty_i}, \qquad i=1, \cdots, k-1,\\
\alpha_k&=\frac{2-\rho^\infty_k}{1+\rho^\infty_k},\qquad\alpha_f=\frac{1}{1+\rho^\infty_k}.
\end{aligned}
\end{equation}
where the eigenvalues of block $i$ approaches to $\rho^\infty_i$ when $\sigma \to \infty$.

\section{Concluding remarks}
We propose a new class of higher-order generalized-$\alpha$ methods that maintain all the attractive features  of the original generalized-$\alpha$ for hyperbolic systems.. In particular, at each time step, we obtain a $2k$ order of accuracy in time by solving $k$ matrix systems consecutively and implicitly. We then update the other $2k$ variables explicitly. The dissipation control is also provided by introducing parameters corresponding to each set of equations.

\section*{Acknowledgement}
This publication was made possible in part by the CSIRO Professorial Chair in Computational Geoscience at Curtin University and the Deep Earth Imaging Enterprise Future Science Platforms of the Commonwealth Scientific Industrial Research Organization, CSIRO, of Australia. Additional support was provided by the European Union's Horizon 2020 Research and Innovation Program of the Marie Sk{\l}odowska-Curie grant agreement No. 777778, the Mega-grant of the Russian Federation Government (N 14.Y26.31.0013), the Institute for Geoscience Research (TIGeR), and the Curtin Institute for Computation. The J. Tinsley Oden Faculty Fellowship Research Program at the Institute for Computational Engineering and Sciences (ICES) of the University of Texas at Austin has partially supported the visits of VMC to ICES. The authors also would
like to acknowledge the contribution of an Australian Government Research Training Program Scholarship in supporting this research.

\bibliographystyle{elsarticle-harv}\biboptions{square,sort,comma,numbers}
\bibliography{ref}

\end{document}